\documentclass[10pt]{article}
\usepackage{amsmath,amsthm,amsfonts}
\usepackage{
hyperref}

\newcommand{\GEM}{\mathsf{GEM}}
\newcommand{\dbeta}{\mathsf{beta}}

\newcommand{\BN}{\mathbb{N}}

\newcommand{\BP}{\mathbb{P}}
\newcommand{\BE}{\mathbb{E}}
\newcommand{\Var}{\operatorname{Var}}
\newcommand{\re}{\mathrm{e}}
\newcommand{\ii}{\mathrm{i}}

\newcommand{\deq}{\overset{d}{{}={}}}
\newcommand{\simas}{\mathrel{\sim_{\mathrm{a.s.}}}}

\renewcommand{\Re}{\operatorname{Re}}

\newtheorem{lemma}{Lemma}
\newtheorem{thm}{Theorem}

\newtheorem{corollary}{Corollary}
\theoremstyle{definition}
\newtheorem*{remark*}{Remark}

\begin{document}

\title{Successive maxima of samples from a GEM distribution}
\author{Jim Pitman\thanks{Statistics Department, 367 Evans Hall \# 3860, University of California, Berkeley, CA 94720-3860, U.S.A.}
\ and
Yuri Yakubovich\thanks{Saint Petersburg State University, St.\;Petersburg State University, 7/9 Universitetskaya nab., St.\;Petersburg, 199034 Russia}}
\maketitle
\begin{abstract}
We show that the maximal value in a size $n$ sample from GEM$(\theta)$ distribution is distributed as a sum of independent geometric random variables. 
This implies that the maximal value grows as $\theta\log(n)$ as $n\to\infty$. For the two-parametric GEM$(\alpha,\theta)$ distribution we show that 
the maximal value grows as a random factor of $n^{\alpha/(1-\alpha)}$ and find the limiting distribution.
\end{abstract}



\section{Introduction}

Consider a sequence of independent and identically distributed (i.i.d.)\ random variables
$X_1,X_2,\dots$. The asymptotic behaviour of a maximum of its finite sample
\begin{equation}\label{eq:defmax}
M_n=\max\{X_1,\dots,X_n\}
\end{equation}
is quite well understood. For continuous distribution the most natural question is the limiting behaviour of 
the sample maximum. The answer is well known: after proper rescaling the distribution of the maximum weakly converges
to a non-degenerate limit which must have the distribution function $F_\alpha(x)=\exp(-(1+x\alpha)^{-1/\alpha})$, for 
some $\alpha\in(-\infty,\infty)$
and $x$ such that $1+x\alpha>0$ (and $F_\alpha(x)$ equals 0 or 1 for other $x$), see, e.g.,~\cite{MR1791071}. Here $\alpha$ (and the scaling) 
depends on the behaviour of the distribution near the supremum of its support.  Of course, limits can be also
degenerate and there exist distributions for which no non-degenerate limit is possible.

The situation is quite different for exchangeable samples, that is the infinite sequence of random variables with a distribution invariant 
under arbitrary finite permutation of indices.  In this case any distribution can appear as a limiting one for the maximum, as shown by the following
example, which seems to be folklore. Let $Z_1,Z_2,\dots$ be a sequence of  i.i.d.\  random variables and $M$ be an independent random variable
with some given distribution. Taking $X_n=Z_n+M$ gives an exchangeable sequence $X_1,X_2,\dots$, and if the distribution of
$Z_1$ has the support bounded above then $M_n=\max\{X_1,\dots,X_n\}$ converges a.s.\ to a shift of $M$ without any rescaling. So the question
about the maximum of an exchangeable sample is not very interesting in general. However in this note we present an example 
when a non-trivial exact distribution of the maximum can be found for a family of exchangeable samples. This family is 
the so-called $\GEM$ distribution described below in Section~\ref{sec:GEM}.  In Section~\ref{sec:maxGEM(theta)}
we give two proofs, analytic and probabilistic, of the following fact, which is further generalized in Section~\ref{sec:Ntheta}. 
Here and below the notation $X\deq Y$ means that random variables
$X$ and $Y$ have the same distribution.

\begin{thm}\label{thm:Mn}
Let $X_1,X_2,\dots$ be an exchangeable sequence obtained by inpependent random sampling from a $\GEM(\theta)$ distribution on the positive integers.  
Then the maximum $M_n$ of $X_1, \ldots X_n$ satisfies
\begin{equation}\label{eq:Mnassum}
M_n-1\deq G_1+\dots+G_n\,,
\end{equation}
where $G_1,\dots,G_n$ are independent geometric random variables with the distributions
\begin{equation}\label{eq:Ggeom}
\BP[G_i=k]=\tau_i(1-\tau_i)^{k} \mbox{ where } \tau_i=\frac{i}{\theta+i} \mbox{ and } k=0,1,2,\dots.
\end{equation}
\end{thm}

An easy consequence is that unlike the independent case the rescaled maximum of a $\GEM(\theta)$ sample has a normal 
limit. 

For the two parameter $\GEM(\alpha,\theta)$ distribution the representation \eqref{eq:Mnassum} as a sum of independent random variables is
no longer valid. Nevertheless we show in Section~\ref{sec:max2} that in this case the maximum of a size $n$ sample behaves as a random multiple of 
$n^{\alpha/(1-\alpha)}$ as $n\to\infty$.

For discrete distributions ties can occur and the question arises how many values
equal to the maximum can occur in the sample.  The answer to this question for independent random variables is also known: Brands et~al.~\cite{MR1294106}
conjectured and Baryshnikov et al.~\cite{MR1340152} soon confirmed that the number of maxima in a sample of $n$ independent discrete
random variables can exhibit just three types of behaviour as $n\to\infty$: either it converges to 1 or 
to $\infty$  in probability, or it does not have a limit. These three cases can be distinguished in terms of the so-called 
\textit{discrete hazard rates} of the distribution of $X_1$, defined as
\begin{equation}\label{eq:hazard}
h_j=\frac{\BP[X_1=j]}{\BP[X_1\ge j]},\qquad j=1,2,\dots.
\end{equation}
(Here we suppose without loss of generality that $X_1$ assumes values $1,2,\dots$.) If $h_j\to0$ as $j\to\infty$, then
the number of maxima converges in probability to 1, and this is the only possibility for convergence to a proper distribution. 
This result was extended to an almost sure (a.s.)\ convergence by Qi~\cite{MR1458007}, who showed that a.s convergence holds if and only if
the series $\sum_j h_j^2$ converges. Later, a more probabilistic proof of this result was given by Eisenberg \cite{MR2493010}
along with some extensions, see Section~\ref{sec:nummax} below.  His results are also formulated in terms of the discrete hazard rates.
These quantities are random and independent for the $\GEM$ distribution, which allows Eisenberg's results to be translated to the exchangeable $\GEM$ case.

Throughout the paper we denote by $\mathbf{1}_A$ the indicator of the set (or event) $A$. For a non-negative integer $k$ we write 
$(a)_k=a(a+1)\dots(a+k-1)$ for the rising factorial. 
 For two sequences of random variables we write $A_n\simas B_n$ as $n\to\infty$
if the limit $A_n/B_n$ exists and equals 1 a.s. The set of natural numbers is denoted $\BN=\{1,2,\dots\}$.

\section{The $\GEM$ distribution}\label{sec:GEM}

Let 
\begin{equation}\label{eq:qasprod}
Y_0=0;\qquad Y_k=1-\prod_{i=1}^{k}(1-H_i),\qquad k\in\BN,
\end{equation}
where $H_1,H_2,\dots$ is a sequence of independent random variables, $H_i$ has $\dbeta(1-\alpha,\theta+i\alpha)$ distribution, 
that is 
\begin{equation}
\label{eq:Y}
\BP[H_i\in dx]=\frac{1}{B(1-\alpha,\theta+i\alpha)} x^{-\alpha}(1-x)^{\theta+i\alpha-1}\mathbf{1}_{\{x\in[0,1]\}},\qquad i\in\BN.
\end{equation} 
Here $\alpha\in[0,1[$ and $\theta>-\alpha$ are real parameters and $B(\cdot,\cdot)$ is Euler's beta function. It is easy to see 
that $Y_k\uparrow 1$ a.s.\ as $k\to\infty$ and hence 
\begin{equation}\label{eq:pi}
\sum_{i=1}^\infty p_i=1,\qquad\text{where }p_i=Y_{i}-Y_{i-1},\quad i\in\BN.
\end{equation}
Thus $(p_i)$ is a random discrete probability distribution, that is a random element of the infinite-dimensional simplex $\{(p_i):p_i\ge 0$ 
and satisfies~\eqref{eq:pi}$\}$.  It is known \cite{MR2245368,MR2663265} as the two parameter $\GEM(\alpha,\theta)$ distribution. 

In the important special case $\alpha=0$ the discrete hazard rates $H_i$ are not only independent but also identically
distributed.  This case is often referred to as the one parameter $\GEM(\theta)$ distribution.  
This case was studied first, following which the two parametric extension proposed by S.~Engen \cite{MR515721} 
has also been extensively studied \cite{MR2245368,MR2663265}.

The $\GEM(\theta)$ distribution enjoys many nice properties, most of which admit some extension to its two parameter generalization
$\GEM(\alpha,\theta)$.  We refer to \cite{MR2245368} for an exposition of the general theory and its applications. We need just the fact that 
the components of the $\GEM(\alpha,\theta)$-distributed vector $(p_i)$ are in size-biased order
\cite[Th.~3.2]{MR2245368}. Recall that the size-biased permutation of a
fixed probability vector $(p_1,p_2,\dots)$ is its random reordering $(p_{\sigma(1)},p_{\sigma(2)},\dots)$
such that $\BP[\sigma(1)=i]=p_i$ for $i\in\BN$, and for each $k\ge1$,  $\BP[\sigma(k+1)=i|\sigma(1)=i_1,\dots,\sigma(k)=i_k]=p_i/(1-p_{i_1}-\dots-p_{i_k})$
for $i\in\BN\setminus\{i_1,\dots,i_k\}$.  The size-biased permutation can be also defined on the space of random discrete distributions
by means of conditioning, and it is easy to see that this operation is idempotent. 

Suppose that the distribution of $\mathbf{Y}=(Y_k)_{k=1}^\infty$ is defined by \eqref{eq:qasprod} and \eqref{eq:Y}
for some $\alpha\in[0,1[$ and $\theta>-\alpha$.
Given $\mathbf{Y}$ 
consider i.i.d.\ random variables $X_1,X_2,\dots$ with values in $\BN$ such that 
\begin{equation}\label{eq:defX}
\BP[X_1\le k|\mathbf{Y}]=Y_k,\qquad k\in\BN.
\end{equation}
We refer to the unconditional distribution of $(X_1,X_2,\dots)$ as the \textit{$\GEM(\alpha,\theta)$ exchangeable distribution} and to 
its finite-dimensional realization $(X_1,\dots,X_n)$ as the \textit{$\GEM(\alpha,\theta)$ exchangeable sample}. Actually,
in most applications
it is not quite natural to suppose that $X_i$ take integer values, the values are usually considered as some classes
of objects which unlike integers have no order structure.  Hence the questions usually asked about the  $\GEM$ samples \cite{MR1860188,MR2206540}
concern the number of distinct values in the sample, the number of values present exactly once
etc. 
%
%
However the invariance of the $\GEM$ distribution under size-biased permutation allows us to give an invariant description of
the $\GEM$ exchangeable sample $X_1$ and of the sample maximum \eqref{eq:defmax}.  

The size-biased permutation can
be obtained by the following construction known as Kingman's paintbox.  Consider a partition of the unit interval $[0,1]$
into intervals, either deterministic or random, and
an independent sequence of i.i.d.\ random variables $\mathbf{V}=(V_1,V_2,\dots)$ uniform on $[0,1]$. Each $V_i$ falls into some interval of the
partition a.s., and we say it discovers a new partition interval if the interval containing $V_i$ contains none of the previous values $V_1,\dots,V_{i-1}$. 
Then rearranging the
intervals in the order of their discovery by the sequence $\mathbf{V}$ gives the size-biased permutation of the partition (or of the probability vector of
its interval lengths).   We call the sequence $\mathbf{V}$ the \textit{uniform sampling sequence}.

Suppose now that the partition is random and the lengths of its intervals are some rearrangement of the $\GEM(\alpha,\theta)$ distribution. 
Take also an additional independent random variable $U_1$ uniform on $[0, 1]$, independent of the uniform sampling sequence $\mathbf{V}$.
The variable  $U_1$ discovers some interval of the partition. Consider the uniform sampling sequence $\mathbf{V}$ term by term until it discovers the same interval. 
Then $X_1$, the first sample from the $\GEM$ distribution, may be represented as the count of distinct intervals discovered by $\mathbf{V}$ until it discovers the interval containing $U_1$,
with the count including the interval containing $U_1$. Consider now the maximum $M_n$ as in \eqref{eq:defmax} of a sample from the $\GEM$ distribution. 
Similarly, a sample $U_1, \ldots U_n $ of i.i.d.\ uniform on $[0, 1]$ random variables, independent of the partition and $\mathbf{V}$,
 discovers some intervals of the partition. Then $M_n$ may be represented as the number of intervals discovered by $\mathbf{V}$ until it discovers all the intervals containing 
$U_1, \ldots, U_n$. This is the invariant description mentioned above, not relying 
on the order structure on the values assumed by the sample, which may be unnatural in some applications.

There is also another interpretation of the $\GEM(\alpha,\theta)$ maximum $M_n$. Let the interval $[0,1]$ be
partitioned by the random points $Y_1,Y_2,\dots$ defined by \eqref{eq:qasprod},~\eqref{eq:Y}. For $0 \le u \le 1$ let 
\begin{equation}\label{eq:defN}
N_{\alpha,\theta}(u):= \sum_{n=1}^\infty 1(Y_n \le u )
\end{equation}
be the point process counting the cut-points in the $\GEM(\alpha,\theta)$ interval partition.
Then
\begin{equation}\label{eq:Mnfromorder}
M_n - 1 = N_{\alpha,\theta} ( U_{n,n})
\end{equation}
where  $0 < U_{n,1} < \cdots < U_{n,n}$ are the usual order statistics of $U_1, \ldots, U_n$.


\section{Maximum of a $\GEM(\theta)$ sample}\label{sec:maxGEM(theta)}

In this section we prove Theorem~\ref{thm:Mn}. We start with an analytic proof 
and then provide a probabilistic proof which works just for the $\GEM(\theta)$ case.
The first proof is based on the connection between the maximum $M_n$ and the moments of the tail probabilities $1-Y_k$. 

\begin{lemma}\label{lem:Mn}
Let $M_n = \max_{1 \le k \le n} X_k$ for a sequence of exchangeable positive integer valued random variables $X_1,\ldots,X_n$ defined as as in \eqref{eq:defX} to be conditionally
independent and identically distributed given some random sequence $\mathbf{Y} = (Y_1, Y_2, \ldots)$ with $0 \le Y_1 \le Y_2 \le \cdots \uparrow 1$ a.s., with
$\BP[X_1\le k|\mathbf{Y}]=Y_k$ for $k = 1,2, \ldots$.  Then the probability generating function
of $M_n-1$ admits the representation
\begin{equation}\label{eq:genfunc}
\BE z^{M_n - 1} = (1-z) \sum_{j=0}^n {n \choose j } (-1)^j 
\sum_{k=1}^\infty \BE  (1 - Y_{k})^j z^{k-1}.
\end{equation}
\end{lemma} 

\begin{proof} 
The probability generating function of any non-negative integer random variable $N$ can be evaluated as
\[
\BE z^N = (1 - z ) \sum_{k = 1}^\infty \BP[N < k ]z^{k-1} .
\]
Applied to $M_n - 1$  this gives \eqref{eq:genfunc} 
because for $k\in\BN$
\[
\smash[b]{\BP[M_n-1< k|\mathbf{Y}]
=\BP[M_n\le k|\mathbf{Y}]
=Y_{k}^n=(1-(1-Y_{k}))^n
=\sum_{j=0}^n{n\choose j}(-1)^j(1-Y_{k})^j\,.}
\]
\end{proof}

\begin{proof}[Proof of Theorem~\ref{thm:Mn}]
For $\GEM(\alpha,\theta)$ we have that $ 1 - H_i$ has $\dbeta (\theta + i \alpha, 1 - \alpha)$ distribution, so
\[
\BE ( 1 - H_i)^j = \frac{B(\theta+i\alpha+j,1-\alpha)}{B(\theta+i\alpha,1-\alpha)}
= \frac{ ( \theta + i \alpha )_j }{ (\theta  + (i-1) \alpha  + 1)_j }
\]
and hence
\[
\BE (1 - Y_k)^j  = \prod_{i= 1}^k \frac{ ( \theta + i \alpha )_j }{ (\theta  + (i-1) \alpha  + 1)_j } 
\]
which can be fed into the generating function \eqref{eq:genfunc}. 
Only in the case $\alpha = 0$ does there seem to be much simplification. Then
\[
\BE (1 - Y_k)^j  = \left( \frac{\theta } { \theta + j } \right)^k
\]
and the series becomes
\[
\sum_{k=1}^\infty \BE  (1 - Y_k)^j z^{k-1} = z^{-1} \sum_{k=1}^\infty  \left( \frac{\theta  z } { \theta + j } \right)^k =  \frac{ \theta }{ j + \theta(1-z)}
\]
hence
\begin{equation}\label{eq:genfM}
\BE z^{M_n - 1 } = ( 1 -z ) \sum_{k=0}^n {n \choose j } \frac{ (-1)^j \theta }{ j + \theta (1 - z ) } = \prod_{i= 1}^n \frac{ i } { i + \theta ( 1 - z) }\,.
\end{equation}
The last equality follows from the well-known partial fraction decomposition
\begin{equation}\label{eq:partfr}
\frac{1}{(x)_{n+1}}=\frac{1}{n!}\sum_{j=0}^n(-1)^j{n\choose j}\frac{1}{x+j}
\end{equation}
which can be verified, for instance, by multiplying \eqref{eq:partfr} by $x+k$ and plugging in $x=-k$ for $k=0,1,\ldots,n$.
Since the factors in the right-hand side of \eqref{eq:genfM} are the probability generating functions of $G_i$ defined by \eqref{eq:Ggeom}, 
the claim  of Theorem~\ref{thm:Mn} follows.
\end{proof}

\medskip

The representation of $M_n$ as a sum of independent geometric random variables invites a direct interpretation of the summands, and 
such an interpretation can be given using the second construction of $M_n$ in terms of the order statistics presented in the end of Section~\ref{sec:GEM}.
We write for short $N_\theta=N_{0,\theta}$ where $N_{\alpha,\theta}$ is defined by \eqref{eq:defN}.

\begin{lemma}\label{lem:Ggeom}
If\/  $0=U_{n,0}<U_{n,1}<\cdots<U_{n,n}$ are the order 
statistics for the uniform sample of size $n$ independent of $N_{\theta}$ then the random variables
\[
G_i := N_{\theta} ( U_{n,n-i+1}) -  N_{\theta} ( U_{n,n-i}),\qquad i=1,\dots,n,
\]
are mutually independent and have the geometric distribution \eqref{eq:Ggeom}.
\end{lemma}

\begin{proof}
The properties of $G_i$ follow from the well known fact  
\cite{MR2245368,MR2351686} 
that $(N_{\theta}(u), 0 \le u \le 1)$ is  an inhomogeneous  Poisson process with intensity $\theta du /(1-u)$ for $0 < u < 1$.
The cumulative intensity measure of $[0,u]$ is
\[
\int_0^u \frac{ \theta d v } {1-v} =  - \theta \log ( 1 - u )\,.
\]
Consequently, the $\GEM(\theta)$ cut points $Y_n$ may be constructed as $Y_n =  1 - \exp( - \gamma_n/\theta )$ where $0 < \gamma_1 < \gamma_2 < \cdots$ are
the points of a standard Poisson process 
\[
N(t):= \sum_{i=1}^\infty \mathbf{1}_{\{ \gamma_i \le t \}}
\]
on $(0,\infty)$ with rate $1$ and i.i.d.\ exponential$(1)$ spacings $\gamma_1, \gamma_2 - \gamma_1, \ldots$.
Now the $T_i:= - \theta \log ( 1 - U_i  )$ are points of an i.i.d.\ random sample from the exponential$(1/\theta)$ distribution of $\theta \gamma_1$.
Let $0=T_{n,0}<T_{n,1}M\cdots<T_{n,n}$ be the order statistics of $T_1, \ldots, T_n$ i.i.d.\ like $\theta \gamma_1$, independent of the Poisson process $N$.
The conclusion follows easily from the well known fact that the successive differences
\[
T_{n,n-i+1} - T_{n,n-i}  \deq \frac{ \theta \gamma_1  } { i }
\]
are independent exponential variables (see, e.g., \cite[Repr.~3.4]{MR1791071}),
and another well known and easily verified fact that if $N$ is a standard Poisson process independent of an exponential variable 
$\theta \gamma_1$, where $\theta >0 $ is fixed, then
$N(\theta \gamma_1)$ is 
geometric$(\tau)$ on $\{0,1,2\ldots \}$ with mean $\tau/(1-\tau) = \theta$,
 corresponding to $\tau = \theta/(\theta + 1)$.
\end{proof}

\smallskip

The representation of Theorem~\ref{thm:Mn} yields an easy

\begin{corollary}\label{cor:CLT} 
After a proper rescaling, the maximum of a sample from $\GEM(\theta)$ distribution of size $n$ has the normal limit:
\[
\frac{M_n-\theta\log n}{\sqrt{\theta\log n}}\overset{d}{\to} N,\qquad n\to\infty,
\]
where $N$ has the standard normal distribution.
\end{corollary}

\begin{proof}
Since $\BE G_i=\theta/i$ and $\Var G_i\sim \theta/i$ as $i\to\infty$, the corollary is an easy application of Lindeberg's limit theorem.
\end{proof}

Looking on the form of \eqref{eq:Mnassum} is tempting to suppose that $G_n$ is the difference $M_n-M_{n-1}$ and is independent of $M_{n-1}$. However this is
not the case, because the new sample $U_{n+1}$ gets into an arbitrary position~$k$ in the order statistics and hence changes 
a value of $G_{n-k}$.  Moreover, unlike the independent case, the successive maxima do not form a Markov chain. Heuristically, this happens because 
knowledge of the history provides some information about the realization of $\mathbf{Y}$. It can be shown, for instance, that
$\BP\bigl[M_1=j,M_2=M_3=k|M_1=j,M_2=k\bigr]$ for $j<k$ depends on $j$, but we omit this calculation.

\section{Some generalization for $\GEM(\theta)$ case}\label{sec:Ntheta}
The results of Theorem~\ref{thm:Mn} and Lemma~\ref{lem:Ggeom} can be generalized as follows.
Instead of $M_n - 1 = N_{\theta} ( U_{n,n})$, consider 
$N_{\theta} (\beta_{n,b})$ 
where as above $N_{\theta}=N_{0,\theta}$ is defined by~\eqref{eq:defN} and
$\beta_{n,b}$ is independent of the $\GEM(\theta)$ cut points $\mathbf{Y}$ and has $\dbeta(n,b)$ density at $u$ proportional to $u^{n-1} (1-u)^{b-1}$ . Then $N_{\theta} (\beta_{n,b}) $
is also distributed as a sum of independent geometric random variables.  

\begin{thm}\label{thm:Gsum2}
For $n\in\BN$ and any $\theta,b>0$ the following equality in distribution holds:
\begin{equation}\label{eq:Gsum2}
N_{\theta} (\beta_{n,b})\deq \sum_{i= 1}^n G_i (b,\theta)
\end{equation}
where the summands are mutually independent and $G_i(b,\theta)$ has geometric$(\tau_i(b,\theta))$ distribution, with
\begin{equation}\label{eq:taubtheta}
\tau_i(b,\theta) := \frac{ b + i - 1 }{ b + i - 1 + \theta } .
\end{equation}
\end{thm}

\begin{proof}
Consider first $N_{\theta} (W)$ where $W$ is a random variable with some arbitrary distribution on $[0,1]$, independent of $N_{\theta}$.
For  $W = u$ fixed, the distribution of $N_{\theta} (u )$ is Poisson$ ( - \theta \log ( 1 - u ) )$ with the probability generating function 
\[
\BE z^{N_{\theta}(u) } =   \exp \left[ - (1 - z ) ( - \theta \log ( 1- u ) \right] = (1 - u ) ^{\theta (1-z)}\,.
\]
For general $W$ the distribution of $N_{\theta} (W )$ ranges over all mixed Poisson distributions. 
Explicitly, the probability generating function of $W$ is
\[
\BE z^{N_{\theta}(W) } =  \BE(1 - W ) ^{\theta (1-z)}.
\]
In particular, if $W = \beta_{a,b}$ has the $\dbeta(a,b)$ distribution then
\begin{align*}
\BE z^{N_{\theta}(\beta_{a,b}) } &{}=  \BE(1 - \beta_{a,b} ) ^{\theta (1-z)} \\
&{}= \BE\beta_{b,a} ^{\theta (1-z)} \\
&{}= \frac{ \Gamma( b + \theta(1-z) )  } { \Gamma( a + b + \theta(1-z) ) } \frac{ \Gamma( a + b ) }{\Gamma(b) }.
\end{align*}
Specifically, if $a = n$ is a positive integer, then 
\[
\frac{ \Gamma(n + b)}{ \Gamma(b) } = (b)_n:= \prod_{i = 1}^n (b + i - 1)
\]
so
\begin{align*}
\BE z^{ N_{\theta}(\beta_{n,b}) } &{}=  \frac{ \Gamma( b + \theta(1-z) ) } { \Gamma( n + b + \theta(1-z) } \frac{ \Gamma( n + b ) }{\Gamma(b) } \\
&{}=  \frac{ (b)_{n}  } { (b + \theta(1-z))_n } \\
&{}=  \prod_{i=1}^{n} \frac{ (b + i - 1) } { (b + i - 1 + \theta(1-z)) } \\
&{}=  \prod_{i=1}^{n} \frac{ \tau_i(b,\theta)   } { ( 1 - (1- \tau_i (b,\theta)) z ) }
\end{align*}
for $\tau_i(n,\theta)$ as in~\eqref{eq:taubtheta}. Since the $i$-th factor is the probability generating function for 
the geomet\-ric($\tau_i(n,\theta)$) distribution, the claim~\eqref{eq:Gsum2} follows.
\end{proof}

\begin{remark*}
Notice that $U_{n,n} \deq \beta_{n,1}$,  
so \eqref{eq:Gsum2} is a generalization of \eqref{eq:Mnassum}.
\end{remark*}

\section{Maximum of a $\GEM(\alpha,\theta)$ sample for $0<\alpha<1$}\label{sec:max2}

The technique of the previous two sections does not seem to work for the case $0<\alpha<1$. However the asymptotics of the $\GEM$
distribution in this case are known sufficiently well to find the asymptotic behaviour of $M_n$ as $n\to\infty$.

The key role in the study of the $\GEM(\alpha,\theta)$ distribution for the case $0<\alpha<1$ is played by the notion
of the $\alpha$-diversity of the exchangeable sample. Denote $K_n$ the number of distinct values in the $\GEM(\alpha,\theta)$ sample
of size $n$. 
It is known \cite[Th.~3.8]{MR2245368}
that if $0<\alpha<1$ and $\theta>-\alpha$ there exists a limit
\[
\lim_{n\to\infty}\frac{K_n}{n^\alpha}=D_{\alpha,\theta}>0
\]
a.s.\ and in $p$\/-th mean for every $p>0$, and the distribution of the limiting random variable $D_{\alpha,\theta}$, known as the $\alpha$-diversity, is
determined by its moments
\begin{equation}\label{eq:diver}
\BE D_{\alpha,\theta}^p=\frac{\Gamma(\theta+1)}{\Gamma(\tfrac{\theta}{\alpha}+1)}\,\frac{\Gamma(p+\tfrac{\theta}{\alpha}+1)}{\Gamma(p\alpha+\theta+1)}\,.
\end{equation}
Moreover, the $\alpha$-diversity $D_{\alpha,\theta}$ is  a.s.\ determined by $\mathbf{Y}$ and 
\[
\BP[X_1 > k|\mathbf{Y}]\simas \alpha D_{\alpha,\theta}^{1/\alpha}\, k^{1-1/\alpha},\qquad k\to\infty,
\]
see \cite[Sec.~10]{MR2318403} or \cite[Lemma~3.11]{MR2245368}.
For such a power law it is well known that the maximum of an independent sample of size $n$ converges in distribution
to the Fr\'echet distribution. Namely, writing for short $\beta=1/\alpha-1$, for any fixed $x>0$
\begin{align*}
\BP\bigl[M_n\le xn^{1/\beta}\bigm|\mathbf{Y}\bigr]&{}=\bigl(1-\BP\bigl[X_1 > xn^{1/\beta}\bigm|\mathbf{Y}\bigr]\bigr)^n\\
&{}\simas \left(1-\alpha D_{\alpha,\theta}^{1/\alpha}\frac{x^{-\beta}}{n}\right)^n\\
&{}\to \exp\bigl(-\alpha D_{\alpha,\theta}^{1/\alpha}x^{-\beta}\bigr),\qquad n\to\infty.
\end{align*}
Hence, by integration with respect to the distribution of the $\alpha$-diversity, we have the following result.

\begin{thm}\label{thm:Mnalpha}
Let $M_n$ be the maximum of a size $n$ $\GEM(\alpha,\theta)$ exchangeable sample with $0<\alpha<1$ and $\theta>-\alpha$. Then
for each $x>0$
\begin{equation}\label{eq:Mndistr}
\BP\bigl[M_n\le xn^{\alpha/(1-\alpha)}\bigr]\to \BE\exp\bigl(-\alpha D_{\alpha,\theta}^{1/\alpha}x^{-(1-\alpha)/\alpha}\bigr) \mbox{ as }  n\to\infty
\end{equation}
where $D_{\alpha,\theta}$ is the random variable with the distribution determined by its moments \eqref{eq:diver}.
\end{thm}

\begin{remark*}
For the case $\alpha=0$ the asymptotics $K_n\simas \theta\log n$ is well known \cite[Sec.~3.3]{MR2245368}, and comparing this with 
Corollary~\ref{cor:CLT} we see that asymptotically $K_n$ and $M_n$ have the same behaviour. For $\alpha>0$ the situation
is different: $K_n$ should be divided by $n^{\alpha}$ to get a proper limit, and $M_n$ grows much faster as a random factor of $n^{\alpha/(1-\alpha)}$.
\end{remark*}

Note that \eqref{eq:Mndistr} expresses the cumulative distribution function of $\lim n^{-\alpha/(1-\alpha)}M_n$ 
evaluated at $x$ as the 
Laplace transform $\BE\bigl[\re^{-y D_{\alpha,\theta}^{1/\alpha}}\bigr]$ evaluated at $y=\alpha x^{-(1-\alpha)/\alpha}$.
Since the moments of $D_{\alpha,\theta}$ given by \eqref{eq:diver} determine its distribution, we can obtain an explicit but clumsy
expression for the limiting distribution function \eqref{eq:Mndistr}. 

\begin{thm}\label{prop:Mnalphaasint}
For the distribution of $D_{\alpha,\theta}$ determined by the moment function \eqref{eq:diver},
\begin{equation}\label{eq:Mndistrasint}
\BE\exp\bigl(-\alpha D_{\alpha,\theta}^{1/\alpha}x^{-(1-\alpha)/\alpha}\bigr)
=\frac{2\alpha^{1-\theta-\alpha}\Gamma(\theta+1)}{\Gamma(\tfrac{\theta}{\alpha}+1)}
x^{(1-\alpha)(\theta/\alpha+1)}\int_0^\infty v^{\theta+2\alpha-1}\re^{-(v^2/\alpha)^\alpha x^{1-\alpha}} J_\theta(2v)\,dv,
\end{equation}
where $J_\theta$ is the Bessel function.
\end{thm}

\begin{proof}
Writing for short $y=\alpha x^{-(1-\alpha)/\alpha}$ we have, for any $c>0$, 
\[
\re^{-yD_{\alpha,\theta}^{1/\alpha}}=\frac{1}{2\pi\ii}\int_{c-\ii\infty}^{c+\ii\infty}\Gamma(s)\bigl(yD_{\alpha,\theta}^{1/\alpha}\bigr)^{-s}ds
\]
because $\re^{-y}$ and $\Gamma(s)$ form the Mellin pair. We refer to \cite{MR0352890}
for the necessary information about Mellin's transform. By analyticity the expression~\eqref{eq:diver} for moments of $D_{\alpha,\theta}$ is valid
also for complex $p$ at least with $\Re p>-1-\tfrac{\theta}{\alpha}$.  
Hence taking expectation and applying Fubini's theorem yields
\begin{equation}\label{eq:LaplaceD}
\BE\bigl[\re^{-yD_{\alpha,\theta}^{1/\alpha}}\bigr]=\frac{1}{2\pi\ii}\frac{\Gamma(\theta+1)}{\Gamma(\tfrac{\theta}{\alpha}+1)}
\int_{c-\ii\infty}^{c+\ii\infty}\Gamma(s)\,\frac{\Gamma(\tfrac{\theta-s}{\alpha}+1)}{\Gamma(\theta-s+1)}y^{-s}ds
\end{equation}
for $0<c<\alpha+\theta$.
Now, $\Gamma(s)/\Gamma(\theta-s+1)$ is the Mellin transform of $y^{-\theta/2}J_\theta(2\sqrt{y})$ in the fundamental strip $0<\Re s<\tfrac{\theta}{2}+\tfrac{3}{4}$ (\cite[II.5.38]{MR0352890}, where there is a misprint in the right bound) and $\Gamma(\tfrac{\theta-s}{\alpha}+1)$ is the 
Mellin transform of $\alpha y^{-\alpha-\theta}\re^{-y^{-\alpha}}$ for $\Re s<\alpha+\theta$, by the standard transformations
of the Mellin pair $\re^{-y}$ and $\Gamma(s)$. Hence 
 their product in the intersection of fundamental strips is the Mellin transform of the multiplicative convolution and for $0<c<\min\{\tfrac{\theta}{2}+\tfrac{3}{4},\alpha+\theta\}$
by the inversion formula 
\[
\frac{1}{2\pi\ii}
\int_{c-\ii\infty}^{c+\ii\infty}\Gamma(s)\,\frac{\Gamma(\tfrac{\theta-s}{\alpha}+1)}{\Gamma(\theta-s+1)}y^{-s}ds
=\alpha\int_{0}^\infty  (y/u)^{-\alpha-\theta}\re^{-(y/u)^{-\alpha}}u^{-\theta/2}J_\theta(2\sqrt{u})\,\frac{du}{u}\,.
\]
Plugging this into \eqref{eq:LaplaceD}, changing the variable $v=\sqrt{u}$ and returning to the variable $x$ yields the result.
\end{proof}

The right-hand side of~\eqref{eq:Mndistrasint} does not seem to allow much simplification for general $\alpha$. For some rational
$\alpha$ Mathematica evaluates this integral in terms of the hypergeometric function. However for $\alpha=1/2$ the integral can be taken explicitly
and leads to a simple expression.  In this case the integral is the Mellin transform of 
$f(v)=\re^{-v\sqrt{2 x}} J_\theta(2v)$ evaluated at $\theta+1$. According to \cite[I.10.7]{MR0352890}
\[
\int_0^\infty v^{s-1}\re^{-v\sqrt{2 x}} J_\theta(2v)\,dv=(2x)^{-(s+\theta)/2}\frac{\Gamma(\theta+s)}{\Gamma(\theta+1)}\,
{}_2F_1\bigl(\tfrac{\theta+s}{2},\tfrac{\theta+s+1}{2};\theta+1;-\tfrac{2}{x}\bigr)\,
\]
and the last expression simplifies for $s=\theta+1$ because
\[
{}_2F_1(a,b;b;z)=\sum_{k=0}^{\infty}\frac{(a)_k}{k!}z^k=\frac{1}{(1-z)^a}
\]
for $|z|<1$ and by analyticity also for all $z$ with $\Re z<1$. Hence
\[
\int_0^\infty v^{\theta}\re^{-v\sqrt{2 x}} J_\theta(2v)\,dv=\frac{\Gamma(2\theta+1)}{\Gamma(\theta+1)}\,
\frac{1}{(2x+4)^{\theta+1/2}}
\]
and plugging it into \eqref{eq:Mndistrasint} gives the following result.

\begin{corollary}
Let $M_n$ be the maximum of a size $n$ $\GEM(\tfrac{1}{2},\theta)$ exchangeable sample with  $\theta>-\tfrac12$. Then
$M_n/n$ converges in distribution as $n\to\infty$ to a random variable with the cumulative distribution function $\bigl(x/(x+2)\bigr)^{\theta+1/2}$.
\end{corollary}

\section{Number of maximal values}\label{sec:nummax}

As it was mentioned in the introduction, the behaviour of the number of maxima in a sample is related to 
the discrete hazard rates. Denote the number of values in the sample equal to the maximum by
\[
L_n=\sum_{j=1}^n 1_{\{X_j=M_n\}}\,.
\]
For distributions with infinite support the only possible limit for $L_n$ is the degenerate distribution in $1$ because $L_n=1$ infinitely often 
(at least each time when the new maximum is reached).  So a natural questions is how big can $L_n$ be for large $n$.
One of the main results of \cite{MR2493010} is the following:

\begin{lemma}[\cite{MR2493010}]\label{lem:limsup} 
Let $X_1,X_2,\dots$ be a sequence of i.i.d.\ random variables with values in $\BN$ and infinitely supported distribution.
Then, for any $\ell\in\BN$, $\BP[\limsup\nolimits_n L_n=\ell]=1$ if and only if\/ $\sum_{j=1}^\infty h_j^\ell=\infty$ and $\sum_{j=1}^\infty h_j^{\ell+1}<\infty$,
where $h_j$ is defined by \eqref{eq:hazard}. If the above series diverge for all $\ell\in\BN$ then $\BP[\limsup\nolimits_n L_n=\infty]=1$.
\end{lemma}

This result has an immediate consequence for samples from the two parameter $\GEM$ distribution.

\begin{thm}
Let $X_1,X_2,\dots$ have the $\GEM(\alpha,\theta)$ exchangeable distribution. 
Then 
\begin{alignat*}{2}&\BP\bigl[\limsup\nolimits_n L_n=1\bigr]=1,\qquad &&\alpha>0;\\
&\BP\bigl[\limsup\nolimits_n L_n=\infty\bigr]=1,\qquad&&\alpha=0. 
\end{alignat*}
\end{thm}

\begin{proof}
If the distribution of $H_i$ is defined by \eqref{eq:Y} then 
\[
\BE[H_i^k]=\frac{B(1-\alpha+k,\theta+i\alpha)}{B(1-\alpha,\theta+i\alpha)}
=\frac{\Gamma(1-\alpha+k)\Gamma(1+(i-1)\alpha+\theta)}{\Gamma(1-\alpha)\Gamma(1+(i-1)\alpha+\theta+k)}\,.
\]
Hence for $\alpha>0$ 
\[
\BE[H_i^k]\sim\frac{\Gamma(1-\alpha+k)}{\Gamma(1-\alpha)}(i\alpha)^{-k},\qquad i\to\infty,
\]
and since $H_i\in[0,1]$ by Kolmogorov's three series theorem the series $\sum H_i^2$ converges a.s. So $\BP[\limsup\nolimits_n L_n=1|(H_i)]=1$ 
by Lemma~\ref{lem:limsup}, and 
also unconditionally.  On the other hand $\BE[H_i^k]$ does not depend on $i$ for $\alpha=0$, so the series $H_i^k$ diverges by the
same theorem and again Lemma~\ref{lem:limsup} implies $\BP[\limsup\nolimits_n L_n=\infty|(H_i)]=1$ and hence unconditionally.
\end{proof}

\bibliographystyle{plain}
\bibliography{gemmax}

\end{document}